\documentclass[12pt,twoside]{amsart}
\usepackage{amssymb,amsmath,amsthm, amscd, enumerate, mathrsfs}
\usepackage{graphicx, hhline}
\usepackage[all]{xy}
\usepackage[dvipdfmx]{hyperref}

\title{Subadjunction for quasi-log canonical pairs and its applications}
\author{Osamu Fujino}
\date{2020/8/31, version 0.21}
\address{Department of Mathematics, Graduate School of Science, 
Osaka University, Toyonaka, Osaka 560-0043, Japan}
\email{fujino@math.sci.osaka-u.ac.jp}

\DeclareMathOperator{\Nqklt}{Nqklt}
\DeclareMathOperator{\Nklt}{Nklt}
\DeclareMathOperator{\Nlc}{Nlc}
\DeclareMathOperator{\mult}{mult}
\DeclareMathOperator{\Exc}{Exc}
\newtheorem{thm}{Theorem}[section]
\newtheorem{lem}[thm]{Lemma}
\newtheorem{prop}[thm]{Proposition}
\newtheorem{cor}[thm]{Corollary}

\theoremstyle{definition}
\newtheorem{step}{Step}
\newtheorem{case}{Case}
\newtheorem{ex}[thm]{Example}
\newtheorem{defn}[thm]{Definition}
\newtheorem{rem}[thm]{Remark}
\newtheorem*{ack}{Acknowledgments}  
\newtheorem*{conventions}{Conventions}       

\newcommand{\Supp}[0]{\operatorname{Supp}}
\newcommand{\Spec}[0]{\operatorname{Spec}}

\makeatletter
    
    \@addtoreset{equation}{section}
\makeatother

\begin{document}

\begin{abstract}
We establish a kind of subadjunction formula for quasi-log canonical 
pairs. As an application, we prove that a connected projective 
quasi-log canonical pair whose quasi-log canonical class 
is anti-ample is simply connected and rationally chain connected. 
We also supplement the cone theorem for quasi-log canonical 
pairs. 
More precisely, we prove that every negative 
extremal ray is spanned by a rational curve. 
Finally, we 
treat the notion of Mori hyperbolicity for quasi-log canonical pairs. 
\end{abstract}

\maketitle 

\tableofcontents

\section{Introduction}\label{a-sec1}

Let $(X, \Delta)$ be a projective log canonical pair and 
let $W$ be a minimal log canonical center of $(X, \Delta)$. 
Then we can find an effective $\mathbb R$-divisor 
$\Delta_W$ on $W$ such that 
$$
(K_X+\Delta)|_W\sim _{\mathbb R} K_W+\Delta_W 
$$ and 
that $(W, \Delta_W)$ is a kawamata log terminal pair. 
This is a famous subadjunction formula for 
minimal log canonical centers (see 
\cite[Theorem 1]{kawamata2} and 
\cite[Theorem 1.2]{fujino-gongyo}) and has already played 
a very important role in the theory of minimal models 
for higher-dimensional algebraic varieties. 
Hence, 
it is very natural and interesting to consider some useful generalizations. 
In this paper, we prove a kind of subadjunction formula 
for (not necessarily minimal) qlc centers of 
quasi-log canonical pairs. We note that 
the notion of quasi-log canonical pairs 
is a generalization of that of log canonical 
pairs. Then we discuss several powerful applications 
of our new subadjunction formula for quasi-log canonical pairs. 

\medskip 

The main purpose of this paper is to establish the 
following theorem, which we call {\em{subadjunction}} 
for qlc strata. 
Theorem \ref{a-thm1.1} is a generalization of \cite[Corollary 1.10]
{fujino-slc-trivial}, where we treat only minimal qlc centers. 
Our proof heavily depends on the structure theorem for normal 
irreducible 
quasi-log canonical 
pairs established in \cite[Theorem 1.7]{fujino-slc-trivial}. 
We note that it is a consequence of some deep results of 
the theory of variations of mixed Hodge structure discussed in 
\cite{fujino-fujisawa}. Therefore, Theorem \ref{a-thm1.1} is 
highly nontrivial. 

\begin{thm}[Subadjunction for qlc strata]\label{a-thm1.1} 
Let $[X, \omega]$ be a quasi-log canonical pair and 
let $W$ be a qlc stratum of $[X, \omega]$. 
Let $\nu:W^\nu\to W$ be the normalization. 
Assume that $W^\nu$ is quasi-projective 
and $H$ is an ample $\mathbb R$-divisor on $W^\nu$. 
Then there exists a boundary $\mathbb R$-divisor 
$\Delta$ on $W^\nu$ such that $$K_{W^\nu}+\Delta\sim 
_{\mathbb R}\nu^*(\omega|_W)+H$$ and 
that $$\Nklt(W^\nu, \Delta)=\nu^{-1}\Nqklt(W, \omega|_W), 
$$ 
where $\Nklt(W^\nu, \Delta)$ denotes 
the non-klt locus of $(W^\nu, \Delta)$.  
More precisely, the equality $$\nu_*\mathcal J(W^\nu, \Delta)=
\mathcal I_{\Nqklt(W, \omega|_W)}$$ holds, where 
$\mathcal J(W^\nu, \Delta)$ is the multiplier ideal 
sheaf of $(W^\nu, \Delta)$ and 
$\mathcal I_{\Nqklt(W, \omega|_W)}$ is the defining ideal 
sheaf of $\Nqklt(W, \omega|_W)$ on $W$. 
Furthermore, if $[X, \omega]$ has a $\mathbb Q$-structure 
and $H$ is an ample $\mathbb Q$-divisor on $W^\nu$, 
then we can make $\Delta$ a $\mathbb Q$-divisor on $W^\nu$ such that  
$$K_{W^\nu}+\Delta\sim _{\mathbb Q} \nu^*(\omega|_W)+H$$ holds. 
\end{thm}

We give two remarks 
in order to help the reader understand Theorem \ref{a-thm1.1}. 

\begin{rem}\label{a-rem1.2}
In Theorem \ref{a-thm1.1}, $[W, \omega|_W]$ naturally becomes 
a quasi-log 
canonical pair by adjunction (see \cite[Theorem 6.3.5 (i)]
{fujino-foundations}) and 
$\Nqklt(W, \omega|_W)$ denotes the union of all qlc 
centers of $[W, \omega|_W]$. 
By adjunction again (see \cite[Theorem 6.3.5 (i)]{fujino-foundations}), 
$$[\Nqklt(W, \omega|_W), \omega|_{\Nqklt(W, \omega|_W)}]$$ 
becomes a quasi-log canonical pair. 
\end{rem}

\begin{rem}\label{a-rm1.3}
By \cite[Theorem 1.1]{fujino-haidong}, we know that 
$[W^\nu, \nu^*(\omega|_W)]$ naturally 
has a quasi-log canonical 
structure. 
In the proof of Theorem \ref{a-thm1.1}, we see that the 
equality $$\mathcal J(W^\nu, \Delta)=\mathcal I_{\Nqklt(W^\nu, 
\nu^*(\omega|_W))}$$ holds, 
where $\mathcal I_{\Nqklt(W^\nu, 
\nu^*(\omega|_W))}$ is the defining ideal sheaf of $\Nqklt(W^\nu, 
\nu^*(\omega|_W))$, the union of all qlc centers of $[W^\nu, 
\nu^*(\omega|_W)]$, on $W$. 
\end{rem}

By combining Theorem \ref{a-thm1.1} with \cite[Theorem 1.2]{fujino-slc}, 
we can easily obtain: 

\begin{cor}[Subadjunction for slc strata]\label{a-cor1.4}
Let $(X, \Delta)$ be a quasi-projective 
semi-log canonical pair and let $W$ be 
an slc stratum of $(X, \Delta)$. 
Let $\nu: W^\nu\to W$ be the normalization and let 
$H$ be an ample $\mathbb R$-divisor on $W^\nu$. 
Then there exits a boundary $\mathbb R$-divisor 
$\Delta^\dag$ on $W^\nu$ such that 
$$K_{W^\nu}+\Delta^\dag\sim _{\mathbb R} \nu^*\left((K_X+\Delta)|_W\right)
+H$$ and that 
$$\Nklt (W^\nu, \Delta^\dag)=\nu^{-1}E, 
$$ 
where $E$ is the union of all slc centers of $(X, \Delta)$ 
that are strictly contained in $W$ and 
$\Nklt(W^\nu, \Delta^\dag)$ denotes the 
non-klt locus of $(W^\nu, \Delta^\dag)$. 
More precisely, the equality $$\nu_*\mathcal J(W^\nu, \Delta^\dag)=
\mathcal I_{\Nqklt(W, \omega|_W)}$$ holds, where 
$\omega:=K_X+\Delta$, 
$\mathcal J(W^\nu, \Delta^\dag)$ is the multiplier ideal 
sheaf of $(W^\nu, \Delta^\dag)$, and 
$\mathcal I_{\Nqklt(W, \omega|_W)}$ is the defining ideal 
sheaf of $\Nqklt(W, \omega|_W)$ on $W$. 
Note that $[X, \omega]$ naturally becomes 
a quasi-log canonical pair and that $[W, \omega|_W]$ 
has a quasi-log canonical 
structure induced from the natural quasi-log canonical 
structure of $[X, \omega]$ by adjunction. 
Furthermore, if $K_X+\Delta$ is $\mathbb Q$-Cartier and 
$H$ is an ample $\mathbb Q$-divisor 
on $W^\nu$, then we can make $\Delta^\dag$ a 
$\mathbb Q$-divisor on $W^\nu$ such that $$K_{W^\nu}+\Delta^\dag\sim 
_{\mathbb Q} \nu^*\left((K_X+\Delta)|_W\right)+H$$ holds. 
\end{cor}

Corollary \ref{a-cor1.4} is a very powerful 
generalization of \cite[Theorem 1]{kawamata2}. 
We give a small remark on Corollary \ref{a-cor1.4} for the reader's 
convenience. 

\begin{rem}[see Remark \ref{a-rem4.1}]\label{a-rem1.5}
If $(X, \Delta)$ is log canonical, equivalently, 
$X$ is normal, 
in Corollary \ref{a-cor1.4}, 
then it is sufficient to assume that $W^\nu$ is quasi-projective. 
We do not need to assume that $X$ is quasi-projective 
when $X$ is normal in Corollary \ref{a-cor1.4}. 
\end{rem}

As an application of Theorem \ref{a-thm1.1}, 
we can prove: 

\begin{thm}[Simply connectedness of qlc Fano pairs]\label{a-thm1.6}
Let $[X, \omega]$ be a projective 
quasi-log canonical pair such that $-\omega$ is ample and that 
$X$ is 
connected. Then $X$ is simply connected, that is, 
the topological fundamental group of $X$ is trivial. 
\end{thm}

Theorem \ref{a-thm1.6}, which is a generalization of 
\cite[Theorem 0.2]{fujino-wenfei}, completely 
confirms a conjecture raised by the author 
(see \cite[Conjecture 1.3]{fujino-pull} and Remark 
\ref{p-rem2.10}). We can also prove: 

\begin{thm}[Rationally chain connectedness of qlc Fano pairs]\label{a-thm1.7} 
Let $[X, \omega]$ be a projective quasi-log canonical 
pair such that $-\omega$ is ample and that 
$X$ is connected. 
Then $X$ is rationally chain connected. 
This means that for arbitrary closed points $x_1, x_2\in X$ there exists 
a connected curve $C\subset X$ which contains 
$x_1$ and $x_2$ such that 
every irreducible component of $C$ is rational.
\end{thm}

Of course, Theorem \ref{a-thm1.7} is a generalization 
of \cite[Corollary 2.5]{fujino-wenfei} by \cite[Theorem 1.2]{fujino-slc} 
(see also Theorem \ref{p-thm2.9}) 
and adjunction for quasi-log canonical pairs (see \cite[Theorem 6.3.5 (i)]
{fujino-foundations}). 

\medskip

From now on, we discuss the cone theorem for 
quasi-log canonical pairs. 
Let $[X, \omega]$ be a quasi-log canonical pair and 
let $\pi:X\to S$ be a projective 
morphism between schemes. 
Then it is well known that the cone theorem 
$$
\overline{NE}(X/S)=\overline{NE}(X/S)_{\omega\geq 0}+\sum R_j
$$
holds, where $R_j$'s are the $\omega$-negative extremal rays 
of the relative Kleiman--Mori cone 
$\overline{NE}(X/S)$. 
For the details, see \cite[Theorem 6.7.4]{fujino-foundations}. 
As an application of Theorem \ref{a-thm1.1}, we obtain: 

\begin{thm}[Lengths of extremal rational curves]\label{a-thm1.8}
Each $\omega$-negative 
extremal ray $R_j$ is spanned by an integral {\em{(}}possibly singular{\em{)}} 
rational 
curve $C_j$ on $X$ such that 
$\pi(C_j)$ is a point and that 
$0<-\omega\cdot C_j \leq 2\dim X$. 
\end{thm}

Theorem \ref{a-thm1.8} is a generalization of \cite[Theorem 1]{kawamata1}. 
Note that Theorem \ref{a-thm1.8} depends 
on \cite[Theorem 1]{kawamata1}. More generally, we have: 

\begin{thm}[{Mori hyperbolicity, see \cite[Theorems 1.2 and 
6.5]{svaldi} and 
Theorem \ref{v-thm7.7}}]
\label{a-thm1.9} 
In Theorem \ref{a-thm1.8}, 
the curve $C_j$ can be so taken that there exist 
a qlc stratum $W$ of $[X, \omega]$ and 
a non-constant morphism 
$
f:\mathbb A^1\to W\setminus \Nqklt(W, \omega|_W)
$ 
such that $C_j\cap (W\setminus \Nqklt(W, \omega|_W))$ contains 
$f(\mathbb A^1)$. 
\end{thm}

We note that 
Theorem \ref{v-thm7.7} is obviously a 
generalization of \cite[Theorems 1.2 and 6.5]{svaldi}. 
By the proof of Theorem \ref{a-thm1.9}, 
we obtain: 

\begin{thm}[{Cone theorem for semi-log canonical pairs, 
see \cite[Theorem 1.19]{fujino-slc}}]\label{a-thm1.10} 
Let $(X, \Delta)$ be a semi-log canonical pair and 
let $\pi:X\to S$ be a projective morphism onto a scheme $S$. 
Then, for each $(K_X+\Delta)$-negative 
extremal ray $R$, 
we can find an slc stratum $W$ of $(X, \Delta)$, 
a non-constant morphism $f:\mathbb A^1\to X$, and 
a possibly singular rational curve $C$ whose numerical 
equivalence class spans $R$ such that 
$f(\mathbb A^1)\subset 
C \cap (W\setminus E)$ holds with 
$0<-(K_X+\Delta)\cdot C\leq 2\dim X$, 
where $E$ is the union of all slc centers of $(X, \Delta)$ that are 
strictly contained in $W$. 
\end{thm}

In \cite{fujino-cone}, we will treat the cone theorem 
for quasi-log schemes which are not necessarily quasi-log 
canonical. 

\medskip 

We summarize the contents of this paper. 
In Section \ref{p-sec2}, we recall some basic definitions. 
In Section \ref{a-sec3}, we review some important 
result in \cite{fujino-slc-trivial}, 
which is the main ingredient of this paper. 
In Section \ref{a-sec4}, 
we prove Theorem \ref{a-thm1.1} and Corollary \ref{a-cor1.4}, 
that is, subadjunction for qlc strata and slc strata, respectively. 
In Section \ref{a-sec5}, we explain how to modify the arguments 
in \cite{fujino-wenfei} to prove Theorems \ref{a-thm1.6} and 
\ref{a-thm1.7}. 
In Section \ref{a-sec6}, we discuss lengths of extremal 
rational curves for qlc pairs (see Theorem \ref{a-thm1.8}). 
In Section \ref{v-sec7}, we treat the notion of 
Mori hyperbolicity for quasi-log canonical pairs. 

\begin{conventions}
We work over $\mathbb C$, the complex number field, 
throughout this paper. 
A {\em{scheme}} means a separated scheme of 
finite type over $\mathbb C$. 
A {\em{variety}} means an integral scheme, that is, 
an irreducible and reduced separated scheme of finite type 
over $\mathbb C$. Let $f:Y\to X$ be a proper birational morphism 
between varieties. Then $\Exc(f)$ denotes the {\em{exceptional locus}} of 
$f$. 
We freely use the basic notation of the minimal model 
program 
as in \cite{fujino-fund} and 
\cite{fujino-foundations}. 
For the details of the theory of quasi-log 
schemes, see \cite[Chapter 6]{fujino-foundations}. 
For the details of semi-log canonical pairs, 
we recommend the reader to 
see \cite{fujino-slc} and \cite{kollar}. 
\end{conventions}

\begin{ack}
The author was partially 
supported by JSPS KAKENHI Grant Numbers 
JP16H03925, JP16H06337. 
He first started to prepare this paper as a continuation of 
\cite{fujino-wenfei}. He thanks Wenfei Liu, 
who is the coauthor of 
\cite{fujino-wenfei}, very much for many useful discussions. 
\end{ack}

\section{Preliminaries}\label{p-sec2}
In this section, let us briefly recall some basic definitions. 
For the details, see \cite{fujino-fund}, 
\cite{fujino-foundations}, and \cite{kollar-rational}. 
We also recommend the reader to see \cite[Section 2]{fujino-slc-trivial} 
for the theory of quasi-log schemes. 

\medskip 

Let us explain singularities of pairs and some related 
definitions. 

\begin{defn}[Singularities of pairs]\label{p-def2.1}
A {\em{normal pair}} $(X, \Delta)$ consists of a normal 
variety $X$ and an $\mathbb R$-divisor 
$\Delta$ on $X$ such that $K_X+\Delta$ is $\mathbb R$-Cartier. 
Let $f\colon Y\to X$ be a projective 
birational morphism from a normal variety $Y$. 
Then we can write 
$$
K_Y=f^*(K_X+\Delta)+\sum _E a(E, X, \Delta)E
$$ 
with 
$$f_*\left(\underset{E}\sum a(E, X, \Delta)E\right)=-\Delta, 
$$ 
where $E$ runs over prime divisors on $Y$. 
We call $a(E, X, \Delta)$ the {\em{discrepancy}} of $E$ with 
respect to $(X, \Delta)$. 
Note that we can define the discrepancy $a(E, X, \Delta)$ for 
any prime divisor $E$ over $X$ by taking a suitable 
resolution of singularities of $X$. 
If $a(E, X, \Delta)\geq -1$ (resp.~$>-1$) for 
every prime divisor $E$ over $X$, 
then $(X, \Delta)$ is called {\em{sub log canonical}} (resp.~{\em{sub 
kawamata log terminal}}). 
We further assume that $\Delta$ is effective. 
Then $(X, \Delta)$ is 
called {\em{log canonical}} and {\em{kawamata log terminal}} 
if it is sub log canonical and sub kawamata log terminal, respectively. 

Let $(X, \Delta)$ be a log canonical pair. If there 
exists a projective birational morphism 
$f:Y\to X$ from a smooth variety $Y$ such that 
both $\Exc(f)$ and  $\Exc(f)\cup \Supp f^{-1}_*\Delta$ are simple 
normal crossing divisors on $Y$ and that 
$a(E, X, \Delta)>-1$ holds for every $f$-exceptional divisor $E$ on $Y$, 
then $(X, \Delta)$ is called {\em{divisorial log terminal}} ({\em{dlt}}, for short). 

Let $(X, \Delta)$ be a normal pair. 
If there exist a projective birational morphism 
$f\colon Y\to X$ from a normal variety $Y$ and a prime divisor $E$ on $Y$ 
such that $(X, \Delta)$ is 
sub log canonical in a neighborhood of the 
generic point of $f(E)$ and that 
$a(E, X, \Delta)=-1$, then $f(E)$ is called a {\em{log canonical center}} of 
$(X, \Delta)$. 
\end{defn}

\begin{defn}[Operations for $\mathbb Q$-divisors and 
$\mathbb R$-divisors]\label{p-def2.2} 
Let $X$ be an equidimensional reduced scheme. 
Note that $X$ is not necessarily regular in codimension one. 
Let $D$ be an $\mathbb R$-divisor (resp.~a $\mathbb Q$-divisor), 
that is, 
$D$ is a finite formal sum $\sum _i d_iD_i$, where 
$D_i$ is an irreducible reduced closed subscheme of $X$ of 
pure codimension one and $d_i$ is a real number (resp.~a rational 
number) for every $i$ 
such that $D_i\ne D_j$ for $i\ne j$. 
We put 
\begin{equation*}
D^{<c} =\sum _{d_i<c}d_iD_i, \quad 
D^{\leq c}=\sum _{d_i\leq c} d_i D_i, \quad 
D^{= 1}=\sum _{d_i= 1} D_i, \quad \text{and} \quad
\lceil D\rceil =\sum _i \lceil d_i \rceil D_i, 
\end{equation*}
where $c$ is any real number and 
$\lceil d_i\rceil$ is the integer defined by $d_i\leq 
\lceil d_i\rceil <d_i+1$. Similarly, we put 
$$
D^{>c} =\sum _{d_i>c}d_iD_i \quad \text{and} \quad 
D^{\geq c}=\sum _{d_i\geq c} d_i D_i 
$$
for any real number $c$. 
Moreover, we put $\lfloor D\rfloor =-\lceil -D\rceil$ and 
$\{D\}=D-\lfloor D\rfloor$. 

Let $D$ be an $\mathbb R$-divisor (resp.~a $\mathbb Q$-divisor) 
as above. 
We call $D$ a {\em{subboundary}} $\mathbb R$-divisor 
(resp.~$\mathbb Q$-divisor) 
if $D=D^{\leq 1}$ holds. 
When $D$ is effective and $D=D^{\leq 1}$ holds, 
we call $D$ a {\em{boundary}} 
$\mathbb R$-divisor (resp.~$\mathbb Q$-divisor). 

Let $\Delta_1$ and $\Delta_2$ be $\mathbb R$-Cartier 
(resp.~$\mathbb Q$-Cartier) divisors on $X$. 
Then $\Delta_1\sim _{\mathbb R} \Delta_2$ (resp.~$\Delta_1\sim 
_{\mathbb Q}\Delta_2$) means that $\Delta_1$ is $\mathbb R$-linearly 
(resp.~$\mathbb Q$-linearly) equivalent to $\Delta_2$. 
\end{defn}

In this paper, we need the notion of {\em{multiplier ideal 
sheaves}}. Although it is well known, 
we recall it here for the reader's convenience. 

\begin{defn}[Multiplier ideal sheaves and non-lc ideal sheaves]\label{p-def2.3} 
Let $X$ be a normal variety
and let $\Delta$ be an effective $\mathbb R$-divisor on $X$ 
such that $K_X+\Delta$ is $\mathbb R$-Cartier. 
Let $f:Y\to X$ be a resolution with
$$
K_Y+\Delta_Y=f^*(K_X+\Delta)
$$
such that $\Supp \Delta_Y$ is a simple normal crossing divisor on $Y$. 
We put
$$
\mathcal J(X, \Delta)=f_*\mathcal O_Y(-\lfloor \Delta_Y\rfloor). 
$$
Then $\mathcal J(X, \Delta)$ is an ideal sheaf on $X$ 
and is known as the {\em{multiplier ideal sheaf}} 
associated to the pair $(X, \Delta)$. 
It is independent 
of the resolution $f:Y\to X$. The closed subscheme $\Nklt(X, \Delta)$ 
defined by $\mathcal J(X, \Delta)$ is called 
the {\em{non-klt locus}} of $(X, \Delta)$. 
It is obvious that $(X, \Delta)$ is kawamata log terminal 
if and only if $\mathcal J(X, \Delta)=\mathcal O_X$. 
Similarly, we put 
$$\mathcal J_{\mathrm{NLC}}(X, \Delta)=f_*\mathcal O_X(-\lfloor 
\Delta_Y\rfloor+\Delta^{=1}_Y)$$
and call it the {\em{non-lc ideal sheaf}} associated to the pair $(X, \Delta)$. 
We can check that it is independent of the resolution $f:Y\to X$. The
closed subscheme $\Nlc(X, \Delta)$ defined by 
$\mathcal J_{\mathrm{NLC}}(X, \Delta)$ 
is called the {\em{non-lc locus}} of 
$(X, \Delta)$. It is obvious that $(X, \Delta)$ is log canonical if and only if 
$\mathcal J_{\mathrm{NLC}}(X, \Delta)=\mathcal O_X$. 

By definition, the natural inclusion 
$$
\mathcal J(X, \Delta)\subset \mathcal J_{\mathrm{NLC}}(X, \Delta)
$$ 
always holds. Therefore, 
we have 
$$
\Nlc(X, \Delta)\subset \Nklt(X, \Delta). 
$$
\end{defn}
 
For the details of $\mathcal J(X, \Delta)$ and $\mathcal J_{\mathrm{NLC}}(X, 
\Delta)$, see \cite{fujino-non-lc}, 
\cite[Section 7]{fujino-fund}, and \cite[Chapter 9]{lazarsfeld}. 

\begin{defn}[Semi-log canonical pairs]\label{p-def2.4} 
Let $X$ be an equidimensional scheme which 
satisfies Serre's $S_2$ condition and 
is normal crossing in codimension one. Let $\Delta$ 
be an effective $\mathbb R$-divisor on $X$ 
such that no irreducible component of $\Supp \Delta$ 
is contained in the singular locus of $X$ and that 
$K_X+\Delta$ is $\mathbb R$-Cartier. 
We say that $(X, \Delta)$ is a {\em{semi-log canonical}} pair if 
$(X^\nu, \Delta_{X^\nu})$ is log canonical 
in the usual sense, where $\nu:X^\nu\to X$ 
is the normalization of $X$ and 
$K_{X^\nu}+\Delta_{X^\nu}=\nu^*(K_X+\Delta)$, 
that is, $\Delta_{X^\nu}$ is the sum of the inverse 
images of $\Delta$ 
and the conductor of $X$. An {\em{slc center}} of $(X, \Delta)$ 
is the $\nu$-image of a log canonical center of $(X^\nu, \Delta_{X^\nu})$. 
An {\em{slc stratum}} of $(X, \Delta)$ 
means either an slc center of $(X, \Delta)$ or an 
irreducible component of $X$. 
\end{defn}

We need the notion of {\em{globally embedded simple normal crossing pairs}} 
for the theory of quasi-log schemes described in \cite[Chapter 6]{fujino-foundations}. 

\begin{defn}[Globally embedded simple normal crossing pairs]\label{p-def2.5} 
Let $Z$ be a simple normal crossing 
divisor on a smooth variety $M$ and let $B$ be 
an $\mathbb R$-divisor 
on $M$ such that 
$Z$ and $B$ have no common irreducible components and 
that the support of $Z+B$ is a simple normal crossing divisor on $M$. In this 
situation, $(Z, B|_Z)$ is called a {\em{globally embedded simple 
normal crossing pair}}.
\end{defn}

Let us quickly recall the definition of 
{\em{quasi-log canonical pairs}}. 

\begin{defn}[Quasi-log canonical pairs]\label{p-def2.6}
Let $X$ be a scheme and let $\omega$ be an 
$\mathbb R$-Cartier divisor (or an $\mathbb R$-line bundle) on $X$. 
Let $f:Z\to X$ be a proper morphism from a globally embedded 
simple normal 
crossing pair $(Z, \Delta_Z)$. If the natural map 
$$\mathcal O_X\to f_*\mathcal O_Z(\lceil -(\Delta_Z^{<1})\rceil)$$ is an 
isomorphism, $\Delta_Z$ is a subboundary $\mathbb R$-divisor, 
and $f^*\omega\sim _{\mathbb R} K_Z+\Delta_Z$ holds, 
then $$\left(X, \omega, f:(Z, \Delta_Z)\to X\right)$$ 
is called a {\em{quasi-log canonical pair}} 
({\em{qlc pair}}, for short). If there is no danger of confusion, 
we simply say that $[X, \omega]$ is a qlc pair. 
We usually call $\omega$ the {\em{quasi-log canonical 
class}} of $[X, \omega]$.  

We say that $\left(X, \omega, f:(Z, \Delta_Z)\to X\right)$ or $[X, \omega]$ 
has a {\em{$\mathbb Q$-structure}} if $\Delta_Z$ is a $\mathbb Q$-divisor, 
$\omega$ is a $\mathbb Q$-Cartier divisor (or a $\mathbb Q$-line bundle), 
and $f^*\omega\sim_{\mathbb Q} K_Z+\Delta_Z$ holds in 
the above definition. 
\end{defn}

We can define {\em{qlc Fano pairs}} as follows. 

\begin{defn}[Qlc Fano pairs]\label{p-def2.7}
Let $[X, \omega]$ be a projective qlc pair such that 
$-\omega$ is ample. 
Then we simply say that $[X, \omega]$ is a 
{\em{qlc Fano pair}}. 
\end{defn}

The notion of {\em{qlc strata}} plays a crucial role in the theory of 
quasi-log schemes. 

\begin{defn}[Qlc strata and qlc centers]\label{p-def2.8}
Let $\left(X, \omega, f:(Z, \Delta_Z)\to X\right)$ be a quasi-log canonical 
pair as in 
Definition \ref{p-def2.6}. 
Let $\nu:Z^\nu\to Z$ be the normalization. 
We put $$K_{Z^\nu}+\Theta=\nu^*(K_Z+\Delta_Z),$$ that is, 
$\Theta$ is the sum of the inverse images of $\Delta_Z$ and 
the singular locus of $Z$. 
Then $(Z^\nu, \Theta)$ is sub log canonical. 
Let $W$ be a log canonical center of $(Z^\nu, \Theta)$ 
or 
an irreducible component of $Z^\nu$. 
Then $f\circ \nu(W)$ is called a {\em{qlc stratum}} of 
$\left(X, \omega, f:(Z, \Delta_Z)\to X\right)$. 
If there is no danger of confusion, we simply call it 
a qlc stratum of 
$[X, \omega]$. 
If $C$ is a qlc stratum of $[X, \omega]$ but is not 
an irreducible component of $X$, then $C$ is called 
a {\em{qlc center}} of $[X, \omega]$. 
The union of all qlc centers of $[X, \omega]$ is denoted by 
$\Nqklt (X, \omega)$ (see \cite[Notation 6.3.10]{fujino-foundations}). 
It is important that 
$$[\Nqklt (X, \omega), \omega|_{\Nqklt(X, \omega)}]$$ 
naturally has 
a quasi-log canonical 
structure induced from $\left(X, \omega, f:(Z, \Delta_Z)\to X\right)$ 
by adjunction (see \cite[Theorem 6.3.5 (i)]
{fujino-foundations}). 
\end{defn}

We recall the main result of \cite{fujino-slc}, 
which makes the theory of quasi-log schemes (see 
\cite[Chapter 6]{fujino-foundations}) 
useful for the study of 
semi-log canonical pairs. 

\begin{thm}[{\cite[Theorem 1.2]{fujino-slc}}]\label{p-thm2.9} 
Let $(X, \Delta)$ be a quasi-projective 
semi-log canonical pair. Then 
$[X, K_X+\Delta]$ becomes a quasi-log canonical pair 
such that $W$ is an slc stratum of $(X, \Delta)$ if and only if 
$W$ is a qlc stratum of $[X, K_X+\Delta]$. 
\end{thm}

For the details of Theorem \ref{p-thm2.9}, 
we recommend the reader to see \cite{fujino-slc}. 

\begin{rem}\label{p-rem2.10} 
By combining Theorem \ref{p-thm2.9} with 
adjunction for quasi-log canonical pairs (see 
\cite[Theorem 6.3.5 (i)]{fujino-foundations}), 
we see that any union $V$ of slc strata of a given quasi-projective 
semi-log canonical pair $(X, \Delta)$ becomes a quasi-log canonical 
pair, that is, $[V, (K_X+\Delta)|_V]$ is a quasi-log canonical pair. 
\end{rem}

We collect some basic properties of qlc strata for the reader's convenience. 

\begin{prop}[Basic properties of qlc strata]\label{p-prop2.11}
Let $[X, \omega]$ be a quasi-log canonical pair. 
Then its qlc strata have the following nice properties.  
\begin{itemize}
\item[(i)] there is a unique minimal {\em{(}}with respect to 
the inclusion{\em{)}} qlc stratum through a given point, 
\item[(ii)] the minimal qlc stratum 
at a given point is normal at that point, and 
\item[(iii)] the intersection of two qlc strata is a union of qlc strata.
\end{itemize} 
If $X$ is additionally a connected projective scheme 
and $-\omega$ is ample, 
that is, $[X, \omega]$ is a connected qlc Fano pair, 
then 
\begin{itemize}
\item[(iv)] any union of qlc strata of $[X, \omega]$ is connected, and 
\item[(v)] there is a unique minimal qlc 
stratum of $[X, \omega]$,  which is normal. 
\end{itemize}
\end{prop}
\begin{proof}[Sketch of Proof of Proposition \ref{p-prop2.11}] 
For (i), (ii), and (iii), see \cite[Theorem 6.3.11]{fujino-foundations}. 
For (iv),  it is sufficient to show that 
$H^0(V, \mathcal O_V)=\mathbb C$ for any union 
$V$ of qlc strata of $[X, \omega]$. 
Since $-\omega$ is ample, 
we have $H^1(X, \mathcal I_V)=0$, where 
$\mathcal I_V$ is the defining ideal sheaf of $V$ on $X$, by 
\cite[Theorem 6.3.5 (ii)]{fujino-foundations}. 
Therefore, the surjection 
$$
\mathbb C=H^0(X, \mathcal O_X)\to H^0(V, \mathcal O_V)\to 0
$$ 
implies $H^0(V, \mathcal O_V)=\mathbb C$. 
Finally, we note that (v) is a direct consequence of (i), (ii), (iii) and (iv).
\end{proof}

For the details of the theory of quasi-log schemes, 
we recommend the reader to see \cite[Chapter 6]{fujino-foundations}. 

\medskip 

We close this section with the definition of {\em{rationally chain 
connected schemes}}. 

\begin{defn}[Rationally chain connected schemes]\label{p-def2.12}
A projective scheme $X$ is {\em{rationally chain 
connected}} if and only if 
for arbitrary closed points $x_1, x_2\in X$ there exists 
a connected curve $C\subset X$ which contains 
$x_1$ and $x_2$ such that 
every irreducible component of $C$ is rational. 
\end{defn}

For the details of rationally chain connected schemes and 
various related topics, 
see \cite{kollar-rational}. 

\section{Quick review of \cite{fujino-slc-trivial}}\label{a-sec3}

In this section, we quickly look at the structure theorem 
for normal irreducible quasi-log canonical pairs obtained in 
\cite{fujino-slc-trivial}. 
Theorem \ref{a-thm3.2} is the main ingredient of this paper. 

\medskip

Let us recall the definition of {\em{potentially nef}} divisors 
in order to explain Theorem \ref{a-thm3.2}. 

\begin{defn}[Potentially nef divisors]\label{a-def3.1}
Let $X$ be a normal variety and let $D$ be a divisor on $X$. 
If there exist a completion $\overline X$ of $X$, 
that is, $\overline X$ is a normal complete 
variety and contains $X$ as a dense Zariski open set, and 
a nef divisor $\overline D$ on $\overline X$ such that 
$D=\overline D|_X$, then $D$ is called 
a {\em{potentially nef}} divisor on $X$. 
A finite $\mathbb R_{>0}$-linear (resp.~$\mathbb Q_{>0}$-linear) 
combination of potentially nef divisors is called 
a {\em{potentially nef}} $\mathbb R$-divisor 
(resp.~$\mathbb Q$-divisor). 
\end{defn}

For the basic properties of potentially nef divisors, 
we recommend the reader to see \cite[Section 2]{fujino-slc-trivial}. 

\medskip

The following theorem will play a crucial role in the theory of quasi-log 
schemes (see \cite{fujino-slc-trivial}, 
\cite{fujino-haidong2}, and \cite{fujino-haidong3}). 

\begin{thm}[{Structure theorem for 
normal irreducible quasi-log canonical pairs, see 
\cite[Theorem 1.7]{fujino-slc-trivial}}]\label{a-thm3.2}
Let $[X, \omega]$ be a quasi-log canonical pair such that 
$X$ is a normal variety. 
Then there exists a projective 
birational morphism 
$p:X'\to X$ from a smooth quasi-projective 
variety $X'$ such that 
$$
K_{X'}+B_{X'}+M_{X'}=p^*\omega, 
$$ 
where $B_{X'}$ is a subboundary $\mathbb R$-divisor 
such that $\Supp B_{X'}$ is a simple normal crossing 
divisor and that $B^{<0}_{X'}$ is 
$p$-exceptional, 
and $M_{X'}$ is a potentially nef $\mathbb R$-divisor on $X'$. 
Furthermore, we can make $B_{X'}$ satisfy 
$p(B^{=1}_{X'})=\Nqklt(X, \omega)$. 

We further assume that $[X, \omega]$ has a $\mathbb Q$-structure. 
Then we can make $B_{X'}$ and $M_{X'}$ $\mathbb Q$-divisors 
in the above statement. 
\end{thm}

In \cite{fujino-slc-trivial}, we introduce the notion of 
{\em{basic slc-trivial fibrations}}, which is a kind of canonical bundle 
formula for reducible schemes. 
Then we prove some fundamental properties 
by using the theory of 
variations of mixed Hodge structure on cohomology 
with compact support (see 
\cite{fujino-fujisawa} and \cite{fujino-fujisawa-saito}). 
Theorem \ref{a-thm3.2} (see \cite[Theorem 1.7]{fujino-slc-trivial}) 
is an application of the main result of \cite{fujino-slc-trivial}, 
that is, \cite[Theorem 1.2]{fujino-slc-trivial}. 

\section{Subadjunction for qlc pairs}\label{a-sec4}

In this section, we prove 
Theorem \ref{a-thm1.1}, which is 
a direct consequence of 
Theorem \ref{a-thm3.2}.  
We note that Corollary \ref{a-cor1.4} follows from 
Theorems \ref{a-thm1.1} and \ref{p-thm2.9} 
(see \cite[Theorem 1.2]{fujino-slc}). 

\medskip

Let us start the proof of Theorem \ref{a-thm1.1}. 

\begin{proof}[Proof of Theorem \ref{a-thm1.1}]
By adjunction (see 
\cite[Theorem 6.3.5 (i)]{fujino-foundations}), 
$[W, \omega|_W]$ is a quasi-log canonical 
pair. By \cite[Theorem 1.1]{fujino-haidong}, we see that 
$[W^\nu, \nu^*(\omega|_W)]$ naturally becomes 
a quasi-log canonical pair such that $\Nqklt(W^\nu, 
\nu^*(\omega|_W))=\nu^{-1}\Nqklt(W, \omega|_W)$ holds. 
More precisely, we obtain that the equality 
$$\nu_*\mathcal I_{\Nqklt (W^\nu, \nu^*(\omega|_W))}
=\mathcal I_{\Nqklt(W, \omega|_W)}$$ holds. 
Note that $\mathcal I_{\Nqklt(W, \omega|_W)}$ 
is the defining ideal sheaf of $\Nqklt(W, \omega|_W)$ on 
$W$ and $\mathcal I_{\Nqklt(W^\nu, \nu^*(\omega|_W))}$ is 
that of $\Nqklt(W^\nu, \nu^*(\omega|_W))$ on $W^\nu$. 
By Theorem \ref{a-thm3.2}, 
there is a projective 
birational morphism $p:W'\to W^\nu$ from a smooth quasi-projective 
variety $W'$ such that 
$$
K_{W'}+B_{W'}+M_{W'}=p^*\nu^*(\omega|_W), 
$$ 
where $B_{W'}$ is a subboundary $\mathbb R$-divisor on $W'$ whose 
support is a simple normal crossing divisor, $B^{<0}_{W'}$ is $p$-exceptional, 
$M_{W'}$ is a potentially nef $\mathbb R$-divisor on $W'$, 
and $p(B^{=1}_{W'})=\Nqklt (W^\nu, \nu^*(\omega|_W))$. 
We may further assume that there is an effective 
$p$-exceptional 
divisor $F$ on $W'$ such that $-F$ is $p$-ample and 
that $\Supp 
F\cup \Supp B_{W'}$ is contained in 
a simple normal crossing divisor on $W'$. Then 
$p^*H-\varepsilon F+M_{W'}$ is semi-ample for any $0<\varepsilon \ll 1$. We 
take a general effective $\mathbb R$-divisor $G$ on $W'$ such that $G\sim 
_{\mathbb R} p^*H-\varepsilon F+M_{W'}$ with $0<\varepsilon \ll 1$, 
$\Supp G \cup \Supp B_{W'}\cup \Supp F$ 
is contained in a simple normal crossing divisor on $W'$, 
and $\lfloor (B_{W'}+\varepsilon F+G)^{\geq 1}\rfloor=B^{=1}_{W'}$. 
Then we have 
\begin{equation*}
\begin{split}
K_{W'}+B_{W'}+M_{W'}+p^*H&=K_{W'}+B_{W'}+\varepsilon 
F +p^*H-\varepsilon F +M_{W'}\\&\sim _{\mathbb R} 
K_{W'}+B_{W'} +\varepsilon F +G. 
\end{split}
\end{equation*}
We put $\Delta:=p_*(B_{W'}+\varepsilon F+G)$. 
By construction, $K_{W^\nu}+\Delta\sim _{\mathbb R}
\nu^*(\omega|_W)+H$. 
Let $\mathcal J(W^\nu, \Delta)$ be the multiplier ideal 
sheaf of $(W^\nu, \Delta)$. 
Then $\mathcal J(W^\nu, \Delta)=p_*\mathcal O_{W'}(-
\lfloor B_{W'}+\varepsilon F+G\rfloor)$ by definition 
(see Definition \ref{p-def2.3}). 
Since the effective part of 
$-\lfloor B_{W'}+\varepsilon F+G\rfloor$ is $p$-exceptional, 
we obtain 
\begin{equation}\label{a-eq4.1}
\begin{split}
\mathcal J(W^\nu, \Delta)&=p_*\mathcal O_{W'}(-
\lfloor B_{W'}+\varepsilon F+G\rfloor)\\ 
&=p_*\mathcal O_{W'}(-\lfloor (B_{W'}+\varepsilon F+G)^{\geq 1}\rfloor)\\
&=p_*\mathcal O_{W'}(-B_{W'}^{=1})\\ 
&=\mathcal I_{\Nqklt (W^\nu, \nu^*(\omega|_W))}. 
\end{split}
\end{equation}
As we saw above, by \cite[Theorem 1.1]{fujino-haidong}, we have the 
equality 
\begin{equation}\label{a-eq4.2}
\nu_*\mathcal I_{\Nqklt (W^\nu, \nu^*(\omega|_W))}=
\mathcal I_{\Nqklt (W, \omega|_W)}. 
\end{equation} 
Therefore, we obtain 
$$
\nu_*\mathcal J(W^\nu, \Delta)=
\nu_*\mathcal I_{\Nqklt (W^\nu, \nu^*(\omega|_W))}
=\mathcal I_{\Nqklt (W, \omega|_W)}
$$ 
by \eqref{a-eq4.1} and \eqref{a-eq4.2}. 
Thus we get 
$$
\mathcal J(W^\nu, \Delta)=\mathcal I_{\Nqklt (W^\nu, \nu^*(\omega|_W))}
=\nu^{-1}\mathcal I_{\Nqklt (W, \omega|_W)}\cdot \mathcal O_{W^\nu}. 
$$
This implies that 
\begin{equation*}
\Nklt (W^\nu, \Delta)
=\Nqklt (W^\nu, \nu^*(\omega|_W))
=\nu^{-1}\Nqklt (W, \omega|_W) 
\end{equation*} holds. 
This is what we wanted. 

When $[X, \omega]$ has a $\mathbb Q$-structure, 
we can make $B_{W'}$ and $M_{W'}$ $\mathbb Q$-divisors 
by Theorem \ref{a-thm3.2}. 
Then it is easy to see that we can make $\Delta$ a $\mathbb Q$-divisor 
on $W^\nu$ such that 
$K_{W^\nu}+\Delta\sim _{\mathbb Q} \nu^*(\omega|_W)+H$ 
if $H$ is an ample $\mathbb Q$-divisor and $[X, \omega]$ has a $\mathbb 
Q$-structure by the above construction of $\Delta$. 
\end{proof}

Corollary \ref{a-cor1.4} easily follows from 
Theorems \ref{a-thm1.1} and \ref{p-thm2.9} 
(see \cite[Theorem 1.2]{fujino-slc}). 

\begin{proof}[Proof of Corollary \ref{a-cor1.4}]
By Theorem \ref{p-thm2.9} (see \cite[Theorem 1.2]{fujino-slc}), 
$[X, K_X+\Delta]$ has a 
natural quasi-log canonical structure which is compatible 
with the original semi-log canonical 
structure of $(X, \Delta)$. Then, by adjunction 
(see \cite[Theorem 6.3.5 (i)]{fujino-foundations}), 
$[W, (K_X+\Delta)|_W]$ 
is quasi-log canonical such that 
$\Nqklt (W, (K_X+\Delta)|_W)=E$ 
(see Remark \ref{p-rem2.10}). 
By Theorem \ref{a-thm1.1}, 
we can take an effective $\mathbb R$-divisor $\Delta^\dag$ on $W^\nu$ 
such that $$K_{W^\nu}+\Delta^\dag\sim _{\mathbb R} 
\nu^*\left((K_X+\Delta)|_W\right)+H$$ 
and that $\Nklt (W^\nu, \Delta^\dag)=\nu^{-1}E$. 
More precisely, $$\nu_*\mathcal J(W^\nu, \Delta^\dag)
=\mathcal I_{\Nqklt (W, \omega|_W)}
$$ holds.  
Of course, by Theorem \ref{a-thm1.1}, 
we can make $\Delta^\dag$ a $\mathbb Q$-divisor 
with $$K_{W^\nu}+
\Delta^\dag\sim _{\mathbb Q} \nu^*\left((K_X+\Delta)|_W\right)
+H$$ if $K_X+\Delta$ and $H$ are both $\mathbb Q$-divisors. 
\end{proof}

\begin{rem}\label{a-rem4.1}
In Corollary \ref{a-cor1.4}, if $(X, \Delta)$ is log canonical, that is, 
$X$ is normal,  
then we do not need the assumption that $X$ is quasi-projective. 
This is because $[X, \omega]$, where $\omega:=K_X+\Delta$, always 
has a natural quasi-log canonical structure that is compatible 
with the original log canonical structure of $(X, \Delta)$ (see 
\cite[6.4.1]{fujino-foundations}). 
We do not need the quasi-projectivity of $X$ to construct the 
quasi-log canonical structure on $[X, \omega]$ when $X$ is normal. 
When $X$ is not normal in Corollary \ref{a-cor1.4}, 
we need the quasi-projectivity 
of $X$ to use Theorem \ref{p-thm2.9} 
(see \cite[Theorem 1.2]{fujino-slc}). 
\end{rem}

\section{On qlc Fano pairs}\label{a-sec5}
In this short section, we explain how to modify the arguments 
in \cite{fujino-wenfei} to prove Theorems \ref{a-thm1.6} and \ref{a-thm1.7}. 
Since this section is independent of the other sections, 
the reader can skip it if he or she is not interested in 
qlc Fano pairs. 

\medskip 

We prepare an important lemma, which is an easy application of 
Theorem \ref{a-thm1.1}. 

\begin{lem}[{see \cite[Lemmas 2.3 and 2.6]{fujino-wenfei}}]\label{a-lem5.1}
Let $W$ be a qlc stratum of a connected qlc Fano pair $[X, \omega]$ and 
let $E$ be the union of all qlc strata that are strictly contained 
in $W$. 
We take the normalization $\nu: W^\nu\to W$ of $W$. 
Let $H$ be an ample Cartier divisor on $X$ and 
let $\varepsilon$ be a sufficiently small positive 
real number. 
Then there exists a boundary $\mathbb R$-divisor 
$\Delta$ on $W^\nu$ such that $K_{W^\nu}+\Delta
\sim _{\mathbb R} \nu^*\left((\omega+\varepsilon H)|_W\right)$, 
$\Nklt(W^\nu, \Delta)=\nu^{-1}E$, and $-(K_{W^\nu}+\Delta)$ 
is ample. 
We note that $\Nklt(W^\nu, \Delta)$ is connected since 
$-(K_{W^\nu}+\Delta)$ is ample. 
\end{lem}

\begin{proof}
We note that $-(\omega+\varepsilon H)$ is ample for any 
sufficiently small positive real number $\varepsilon$. 
By Theorem \ref{a-thm1.1}, 
we can take a boundary $\mathbb R$-divisor 
$\Delta$ on $W^\nu$ with $K_{W^\nu}+\Delta\sim 
_{\mathbb R} \nu^*\left((\omega+\varepsilon H)|_W\right)$ and 
$\Nklt(W^\nu, \Delta)=\nu^{-1}E$. 
By the Nadel vanishing theorem (see \cite[Theorem 3.4.2]{fujino-foundations}), 
we have $H^1(W^\nu, \mathcal J(W^\nu, \Delta))=0$, 
where $\mathcal J(W^\nu, \Delta)$ is the multiplier ideal 
sheaf of 
$(W^\nu, \Delta)$, since $-(K_{W^\nu}+\Delta)$ is ample. 
This implies that $\Nklt(W^\nu, \Delta)=\nu^{-1}E$ is connected. 
\end{proof}

By the standard argument 
in the recent developments of the 
theory of higher-dimensional minimal models, 
we have the following lemma. 

\begin{lem}\label{a-lem5.2} 
Let $V$ be a normal projective variety and let $\Delta$ be a boundary 
$\mathbb R$-divisor on $V$ such that $-(K_V+\Delta)$ is ample. 
Then we can take a boundary $\mathbb Q$-divisor 
$\Delta'$ on $V$ such that 
$-(K_V+\Delta')$ is ample and 
that the equality $\mathcal J(V, \Delta')=\mathcal J(V, \Delta)$ holds, 
where $\mathcal J(V, \Delta)$ {\em{(}}resp.~$\mathcal J(V, \Delta')${\em{)}} 
is the multiplier ideal sheaf of 
$(V, \Delta)$ {\em{(}}resp.~$(V, \Delta')${\em{)}}. 
Moreover, we can choose $\Delta'$ such that 
$\mult _P\Delta'=\mult _P\Delta$ holds for any prime divisor $P$ on $V$ with 
$\mult _P \Delta\in \mathbb Q$. 
\end{lem}
\begin{proof}
By slightly perturbing the coefficients of $\Delta$, 
we get a boundary $\mathbb Q$-divisor $\Delta'$ with 
the desired properties. 
We leave the details as an exercise for the reader. 
\end{proof}

Since many results were formulated and stated only for 
$\mathbb Q$-divisors in the literature, Lemma \ref{a-lem5.2} 
is useful and helpful. 

\medskip 

By Lemmas \ref{a-lem5.1} and \ref{a-lem5.2}, 
the proof of \cite[Corollary 2.5 and Theorem 2.7]{fujino-wenfei} 
works with some minor modifications. 

\begin{proof}[Sketch of Proof of Theorems \ref{a-thm1.6} and \ref{a-thm1.7}] 
Let $W_0$ be the unique minimal qlc stratum of $[X, \omega]$ 
(see Proposition \ref{p-prop2.11} (v)). 
Then we can take a boundary $\mathbb Q$-divisor 
$\Delta_0$ on $W_0$ such that 
$(W_0, \Delta_0)$ is kawamata log terminal and 
$-(K_{W_0}+\Delta_0)$ is ample by 
Lemmas \ref{a-lem5.1} and \ref{a-lem5.2}. 
Thus it is well known that $W_0$ is rationally (chain) connected 
and simply connected (see, for example, 
\cite[Corollary 2.4]{fujino-wenfei}). 

Let $W$ be any qlc stratum of $[X, \omega]$. 
By Lemmas \ref{a-lem5.1} and \ref{a-lem5.2}, 
the proof of \cite[Corollary 2.5]{fujino-wenfei} works 
with some minor modifications. 
Therefore, we obtain that $X$ is rationally chain connected. 
Hence we get Theorem \ref{a-thm1.7}. 

By Lemmas \ref{a-lem5.1} and \ref{a-lem5.2} again, 
we can easily see that the proof of \cite[Theorem 2.7]{fujino-wenfei} 
works with some minor changes. 
Hence we see that $X$ is simply connected. 
This is Theorem \ref{a-thm1.6}. 
\end{proof}

\section{Lengths of extremal rational 
curves for qlc pairs}\label{a-sec6}
In this section, we prove Theorem \ref{a-thm1.8}, 
which is a generalization of \cite[Theorem 1]{kawamata1}. 
Our proof of Theorem \ref{a-thm1.8} below heavily depends on 
\cite[Theorem 4.6.7]{fujino-foundations}. 

\medskip 

Let us start the proof of Theorem \ref{a-thm1.8}. 
Although we will treat a more general result 
in Section \ref{v-sec7}, 
the proof of Theorem \ref{a-thm1.8} plays a crucial 
role. 

\begin{proof}[Proof of Theorem \ref{a-thm1.8}]
Let $\varphi_{R_j}: X\to Y$ be the extremal contraction 
associated to $R_j$ (see 
\cite[Theorems 6.7.3 and 6.7.4]{fujino-foundations}). 
By replacing $\pi:X\to S$ with $\varphi_{R_j}: X\to Y$, we may 
assume that 
$-\omega$ is $\pi$-ample. 
We take a qlc stratum $W$ of $[X, \omega]$ such that 
$\pi:\Nqklt(W, \omega|_W)\to \pi\left(\Nqklt(W, \omega|_W)\right)$ is 
finite and 
that $\pi:W\to \pi(W)$ is not finite. 
It is sufficient to find a rational curve $C$ on $W$ such that 
$\pi(C)$ is a point 
and that $0<-(\omega|_W)\cdot C\leq 2\dim W\leq 2\dim X$. 
Therefore, by replacing $\pi:X\to S$ with $\pi:W\to S$, 
we may assume that $X$ is irreducible and that $\pi:
\Nqklt(X, \omega)\to \pi\left(\Nqklt(X, \omega)\right)$ is finite. 
Let $\nu:X^\nu\to X$ be the normalization. 
Then, by \cite[Theorem 1.1]{fujino-haidong}, 
$[X^\nu, \nu^*\omega]$ naturally becomes a 
quasi-log canonical 
pair with $\Nqklt(X^\nu, \nu^*\omega)=\nu^{-1}
\Nqklt(X, \omega)$. 
Therefore, by replacing $\pi:X\to S$ with 
$\pi\circ \nu: X^\nu\to S$, we may 
assume that $X$ is a normal 
variety such that 
$\pi:\Nqklt(X, \omega)\to \pi\left(\Nqklt(X, \omega)\right)$ 
is finite. 
In this situation, 
all we have to do is to find a rational curve $C$ on $X$ such that 
$\pi(C)$ is a point and that 
$0<-\omega\cdot C\leq 2\dim X$. 
Without loss of generality, we may assume that 
$X$ and $S$ are quasi-projective by shrinking $S$ 
suitably. 
Let $H$ be an ample Cartier divisor on $X$. 
By Theorem \ref{a-thm1.1}, we can 
construct a boundary $\mathbb R$-divisor 
$\Delta_{\varepsilon}$ on $X$ such that 
$K_X+\Delta_{\varepsilon}\sim _{\mathbb R} \omega+\varepsilon H$ and 
that $\Nklt(X, \Delta_{\varepsilon})=\Nqklt(X, \omega)$ for 
every positive real number $\varepsilon$. 
Note that $\Nlc(X, \Delta_{\varepsilon}) \subset 
\Nklt(X, 
\Delta_{\varepsilon}) =\Nqklt(X, \omega)$, 
where $\Nlc(X, \Delta_{\varepsilon})$ denotes the non-lc locus 
of $(X, \Delta_{\varepsilon})$ as in Definition \ref{p-def2.3}. 
Therefore, $\pi:\Nlc (X, 
\Delta_{\varepsilon})\to \pi\left(\Nlc(X, \Delta_{\varepsilon})
\right)$ is finite. We assume that 
$\varepsilon$ is sufficiently small 
such that 
$-(\omega+\varepsilon H)$ is 
$\pi$-ample. Then, by the cone theorem 
for $(X, \Delta_{\varepsilon})$, we can 
find a rational curve $C_{\varepsilon}$ on $X$ such that 
$\pi(C_{\varepsilon})$ is a point and 
that $0<-(K_X+\Delta_{\varepsilon})\cdot 
C_{\varepsilon}\leq 2\dim X$ (see 
\cite[Theorem 1.1]{fujino-fund} and 
\cite[Theorem 4.6.7]{fujino-foundations}). 
We take an ample $\mathbb Q$-divisor $A$ 
on $X$ such that $-(\omega+A)$ is $\pi$-ample. 
We take $\{\varepsilon_i\}_{i=0}^\infty$ such that
$\lim _{i\to 0} \varepsilon _i=0$, $\varepsilon_i$ is a 
positive real number, 
and $-(\omega+A+\varepsilon _i H)$ is $\pi$-ample for every $i$. 
As we saw above, we can take a rational curve $C_i$ on $X$ such that 
$\pi(C_i)$ is a point and that 
$0<-(\omega+\varepsilon _i H)\cdot C_i\leq 2\dim X$ for every $i$. 
Note that 
$$
0<A\cdot C_i =\left((\omega+\varepsilon _i H+A)-(\omega+\varepsilon _iH)
\right) \cdot C_i <2\dim X. 
$$ 
It follows that the curves $C_i$ belong to a bounded family. 
Thus, possibly passing to a subsequence, we may assume that 
$C_i=C$ is constant. 
Therefore, we get 
$$
0<-\omega\cdot C=\lim _{i\to \infty} -(\omega+\varepsilon_i H)\cdot 
C=\lim _{i\to \infty} -(\omega+\varepsilon_i H)\cdot 
C_i  \leq 2\dim X. 
$$ 
This is what we wanted. 
\end{proof}

\begin{rem}\label{a-rem6.1}
We expect that the estimate $\leq 2\dim X$ should be 
replaced by $\leq \dim X+1$ in Theorem \ref{a-thm1.8} 
(see \cite[Remark 4.6.3]{fujino-foundations}). 
\end{rem}

We give a remark on the proof of \cite[Theorem 4.6.7]{fujino-foundations}, 
which was used in the proof of Theorem \ref{a-thm1.8} above. 

\begin{rem}\label{a-rem6.2}
In the proof of \cite[Theorem 4.6.7]{fujino-foundations}, 
the author claims that $\pi$ is an 
isomorphism in a neighborhood of $\Nlc(X, \Delta)$ by replacing $\pi:X\to 
S$ with the extremal contraction $\varphi_R: X\to Y$ over $S$. 
However, it is not correct. 
In general, $\pi$ is 
not necessarily an isomorphism 
around $\Nlc (X, \Delta)$ (see Example \ref{a-ex6.3} below). By 
replacing $\pi:X\to S$ with $\varphi_R: X\to Y$, we can assume that 
$\pi: \Nlc (X, \Delta)\to \pi\left(\Nlc (X, \Delta)\right)$ is finite.  
Note that all we need in the proof of \cite[Theorem 4.6.7]{fujino-foundations} 
is the fact that $\pi$ contracts no curves in $\Nlc (X, \Delta)$. 
Therefore, the proof of \cite[Theorem 4.6.7]{fujino-foundations} 
works without any modifications. 
\end{rem}

\begin{ex}\label{a-ex6.3} 
We put $X=\mathbb P^1$, $\pi: X\to S=\Spec \mathbb C$, 
and $\Delta=\frac{3}{2} P$, where $P$ is a point of $X=\mathbb P^1$. 
Then $-(K_X+\Delta)$ is $\pi$-ample and $\rho(X/S)=1$. 
Of course, $\pi$ is not an isomorphism 
around $P=\Nlc (X, \Delta)$. 
\end{ex}

We close this section with an important remark. 

\begin{rem}\label{a-rem6.4}
The proof of \cite[Theorem 4.6.7]{fujino-foundations} 
needs Mori's bend and break technique to 
create rational curves (see \cite[Remark 4.6.4]{fujino-foundations}). 
Therefore, we need the mod $p$ reduction technique for the proof of 
Theorem 
\ref{a-thm1.8}. 
We note that we take a dlt blow-up (see 
\cite[Theorem 4.4.21]{fujino-foundations}) 
in the proof of \cite[Theorem 4.6.7]{fujino-foundations}. 
This means that Theorem \ref{a-thm1.8} depends on 
the minimal model program mainly due to \cite{bchm}. 
\end{rem}

\section{Mori hyperbolicity for quasi-log canonical pairs}\label{v-sec7}

In this final section, we generalize the main result of \cite{svaldi} 
for quasi-log canonical pairs. 
We note that in \cite{svaldi} 
the notion of {\em{crepant log structures}}, 
which is a very special case of that of 
quasi-log schemes, plays a 
crucial role. On the other hand, we can directly 
treat highly singular reducible schemes by the framework of 
quasi-log schemes (see 
\cite[Chapter 6]{fujino-foundations}) 
and basic slc-trivial fibrations (see \cite{fujino-slc-trivial}). 
This is the main difference between \cite{svaldi} and 
our approach here. 

\medskip 

Let us start with the following key result. 

\begin{prop}[{\cite[Proposition 5.2]{svaldi}}]\label{v-prop7.1}
Let $\pi:X\to S$ be a projective morphism from a normal 
$\mathbb Q$-factorial 
variety $X$ onto a scheme $S$. Let $\Delta=\sum _i d_i D_i$ be an effective 
$\mathbb R$-divisor on $X$, where the $D_i$'s are the distinct 
prime components of $\Delta$ for all $i$, such that 
$$\left(X, \Delta':=\sum _{d_i<1}d_iD_i+\sum _{d_i\geq 1} D_i\right)$$ is dlt. 
Assume that 
$(K_X+\Delta)|_{\Nklt(X, \Delta)}$ is nef over $S$. 
Then $K_X+\Delta$ is nef over $S$ or 
there exists a non-constant morphism 
$f:\mathbb A^1\to X\setminus \Nklt(X, \Delta)$ such that 
$\pi\circ f(\mathbb A^1)$ is a point. 

More precisely, the curve $C$, the closure of $f(\mathbb A^1)$ in 
$X$, is a {\em{(}}possibly singular{\em{)}} rational curve 
with $$
0<-(K_X+\Delta)\cdot C\leq 2\dim X. 
$$
\end{prop}

This is one of the most important results of \cite{svaldi}. 
We give a detailed proof for the sake of completeness. 

\begin{proof}[Proof of Proposition \ref{v-prop7.1}]
Note that $\Nklt(X, \Delta)$ coincides with $(\Delta')^{=1}=\lfloor 
\Delta'\rfloor$, 
$\Delta^{\geq 1}$, and $\lfloor \Delta\rfloor$ set theoretically 
because $(X, \Delta')$ is dlt by assumption. 
It is sufficient to construct a non-constant morphism 
$f:\mathbb A^1\to X\setminus \Nklt(X, \Delta)$ such that 
$\pi\circ f(\mathbb A^1)$ is a point with the desired 
properties when 
$K_X+\Delta$ is not nef over $S$. 
By shrinking $S$ suitably, we may assume that $S$ and $X$ are 
both quasi-projective. 
By the cone and contraction theorem (see \cite[Theorem 
1.1]{fujino-fund}), we can take a $(K_X+\Delta)$-negative 
extremal ray $R$ of $\overline {NE}(X/S)$ and the 
associated extremal contraction morphism 
$\varphi:=\varphi_R:X\to Y$ over $S$ since 
$(K_X+\Delta)|_{\Nklt(X, \Delta)}$ is nef over $S$. 
Note that $(K_X+\Delta^{<1})\cdot R<0$ and 
$(K_X+\Delta')\cdot R<0$ hold because 
$(K_X+\Delta)|_{\Nklt(X, \Delta)}$ is nef over $S$. 
Since $(X, \Delta^{<1})$ is kawamata log terminal and $-(K_X+\Delta^{<1})$ is 
$\varphi$-ample, 
we get $R^i\varphi_*\mathcal O_X=0$ for every $i>0$ 
by the relative Kawamata--Viehweg vanishing theorem 
(see \cite[Corollary 5.7.7]{fujino-foundations}). 
By construction, $\varphi: \Nklt(X, \Delta)\to \varphi(\Nklt(X, \Delta))$ is 
finite. 
We have the following 
short exact sequence  
$$
0\to \mathcal O_X(-\lfloor \Delta'\rfloor)\to 
\mathcal O_X\to \mathcal O_{\lfloor \Delta'\rfloor}\to 0. 
$$ 
Since $-\lfloor \Delta'\rfloor -(K_X+\{\Delta'\})=-(K_X+\Delta')$ 
is $\varphi$-ample and $(X, \{\Delta'\})$ is kawamata log terminal, 
$R^i\varphi_*\mathcal O_X(-\lfloor \Delta'\rfloor)=0$ holds 
for every $i>0$ by the 
relative Kawamata--Viehweg vanishing theorem again (see 
\cite[Corollary 5.7.7]{fujino-foundations}).   
Therefore, 
$$
0\to \varphi_*\mathcal O_X(-\lfloor \Delta'\rfloor)\to 
\mathcal O_Y\to \varphi_*\mathcal O_{\lfloor \Delta'\rfloor}\to 0
$$ 
is exact. 
This implies that 
$\Supp \lfloor \Delta'\rfloor=\Supp \Delta^{\geq 1}$ is connected 
in a neighborhood of any fiber of $\varphi$. 
\begin{case}\label{v-7.1case1}
Assume that $\varphi$ is a Fano contraction, that is, 
$\dim Y<\dim X$. 
Then we see that $\Delta^{\geq 1}$ is $\varphi$-ample 
and that $\dim Y=\dim X-1$. Note that 
$\Supp \Delta^{\geq 1}$ is finite over $Y$. 
In this situation, we can easily see that 
every general fiber is $\mathbb P^1$ 
by $R^1\varphi_*\mathcal O_X=0$. 
Moreover, any general fiber intersects 
$\Delta^{\geq 1}$ in at most one point 
by the connectedness of $\Supp \Delta^{\geq 1}$ discussed 
above. 
Therefore, we can find a non-constant morphism 
$f:\mathbb A^1\to X\setminus \Nklt(X, \Delta)$ such that 
$\pi\circ f(\mathbb A^1)$ is a point and that 
$0<-(K_X+\Delta)\cdot C\leq 2\dim X$ holds, 
where $C$ is the closure of $f(\mathbb A^1)$ in $X$. 
\end{case}
\begin{case}\label{v-7.1case2}
Assume that $\varphi$ is a birational contraction and that 
the exceptional locus $\Exc(\varphi)$ of $\varphi$ 
is disjoint from $\Nklt(X, \Delta)$. 
In this situation, we can find a rational curve $C$ in a fiber of $\varphi$ 
with $0<-(K_X+\Delta)\cdot C\leq 2\dim X$ by the 
cone theorem (see \cite[Theorem 1.1]{fujino-fund}). 
It is obviously disjoint from $\Nklt(X, \Delta)$. 
Therefore, we can take a non-constant morphism 
$f:\mathbb A^1\to X\setminus \Nklt(X, \Delta)$ 
such that the closure of $f(\mathbb A^1)$ is $C$. 
\end{case}
\begin{case}\label{v-7.1case3}
Assume that $\varphi$ is a birational 
contraction and that 
$\Exc(\varphi)\cap \Nklt(X, \Delta)\ne \emptyset$. 
In this situation, as in Case \ref{v-7.1case1}, 
we see that $\Delta^{\geq 1}$ is $\varphi$-ample and 
that $\dim \varphi^{-1}(y)\leq 1$ for every $y\in Y$. 
By taking a complete intersection of general hypersurfaces 
of $Y$ and its inverse image, we can reduce the problem to the case 
where $\varphi(\Exc(\varphi))=:P$ is a point. 
Then $R^1\varphi_*\mathcal O_X=0$ implies that 
every irreducible component of $\varphi^{-1}(P)$ 
is $\mathbb P^1$. 
We take any irreducible component $C$ of $\varphi^{-1}(P)$. 
By the connectedness of $\Supp \Delta^{\geq 1}$ discussed 
above, $C$ intersects $\Delta^{\geq 1}$ in at most one point. 
Therefore, we can get a non-constant morphism 
$f:\mathbb A^1\to X\setminus \Nklt(X, \Delta)$ such that 
$f(\mathbb A^1)\subset C\cap (X\setminus \Nklt(X, \Delta))$. 
By applying the cone theorem (see 
\cite[Theorem 1.1]{fujino-fund}) to $\varphi:X\to Y$, 
we may assume that $0<-(K_X+\Delta)\cdot C\leq 2\dim X$. 
\end{case} 
Therefore, we get the desired statement. 
\end{proof}

Let us recall the following useful lemma, which is a kind of dlt blow-ups. 
Here we need the minimal model theory mainly due to 
\cite{bchm}. 

\begin{lem}[{\cite[Theorem 3.4]{svaldi}}]\label{v-lem7.2}
Let $X$ be a normal quasi-projective variety and 
let $\Delta$ be a boundary 
$\mathbb R$-divisor on $X$ such that 
$K_X+\Delta$ is $\mathbb R$-Cartier. 
Then we can construct a projective 
birational morphism 
$g:Y\to X$ from a normal $\mathbb Q$-factorial variety 
$Y$ with the following properties. 
\begin{itemize}
\item[(i)] $K_Y+\Delta_Y:=g^*(K_X+\Delta)$, 
\item[(ii)] the pair $$\left(Y, \Delta'_Y:=\sum _{d_i <1}d_i D_i+
\sum _{d_i\geq 1}D_i\right)$$ 
is dlt, where $\Delta_Y=\sum _i d_i D_i$ is 
the irreducible decomposition of 
$\Delta_Y$, 
\item[(iii)] every $g$-exceptional prime divisor 
is a component of $(\Delta'_Y)^{=1}$, and 
\item[(iv)] $g^{-1}\Nklt(X, \Delta)$ coincides with 
$\Nklt(Y, \Delta_Y)$ and $\Nklt(Y, \Delta'_Y)$ set theoretically. 
\end{itemize}
\end{lem}

\begin{proof}[Sketch of Proof of Lemma \ref{v-lem7.2}]
It is well known that there exists a dlt blow-up $\alpha:Z\to X$ with 
$K_Z+\Delta_Z:=\alpha^*(K_X+\Delta)$ satisfying 
(i), (ii), and (iii) (see \cite[Theorem 4.4.21]{fujino-foundations}). 
Note that $(Z, \Delta^{<1}_Z)$ is a $\mathbb Q$-factorial 
kawamata log terminal pair. 
We take a minimal model $(Z', \Delta^{<1}_{Z'})$ of 
$(Z, \Delta^{<1}_Z)$ over $X$ by \cite{bchm}. 
$$
\xymatrix{
Z \ar[dr]_-\alpha\ar@{-->}[rr]^-\varphi& & Z' \ar[dl]^-{\alpha'}\\ 
&X &
}
$$
Then $K_{Z'}+\Delta^{<1}_{Z'}
\sim _{\mathbb R} -\Delta^{\geq 1}_{Z'}+\alpha'^*(K_X+\Delta)$ 
is nef over $X$. 
Of course, we put $\Delta_{Z'}=\varphi_*\Delta_Z$. 
We take a dlt blow-up 
$\beta:Y\to Z'$ of $(Z', \Delta^{<1}_{Z'}+\Supp \Delta^{\geq 1}_{Z'})$ again 
(see \cite[Theorem 4.4.21]{fujino-foundations}) and put $g:=
\alpha'\circ \beta:Y\to X$. 
It is not difficult to see that this birational morphism $g:Y\to X$ with 
$K_Y+\Delta_Y:=g^*(K_X+\Delta)$ satisfies the desired properties. 
It is obvious that $g^{-1}\Nklt(X, \Delta)$ contains the 
support of $\beta^*\Delta^{\geq 1}_{Z'}$. 
Since $-\beta^*\Delta^{\geq 1}_{Z'}$ is nef 
over $X$, we see that $\beta^*\Delta^{\geq 1}_{Z'}$ coincides with 
$g^{-1}\Nklt(X, \Delta)$ set theoretically. 
For the details, see the proof of \cite[Theorem 3.4]{svaldi}. 
\end{proof}

By combining Proposition \ref{v-prop7.1} with 
Lemma \ref{v-lem7.2}, we obtain: 

\begin{cor}[{\cite[Corollary 5.3]{svaldi}}]\label{v-cor7.3}
Let $X$ be a normal variety and let $\Delta$ 
be a boundary $\mathbb R$-divisor 
on $X$ such that $K_X+\Delta$ is $\mathbb R$-Cartier. 
Let $\pi:X\to S$ be a projective morphism 
onto a scheme $S$. 
Assume that there is no non-constant 
morphism $f:\mathbb A^1\to X\setminus 
\Nklt(X, \Delta)$ such that  
$\pi\circ f(\mathbb A^1)$ is a point. 
Then $K_X+\Delta$ is nef over $S$ if and only if 
$(K_X+\Delta)|_{\Nklt(X, \Delta)}$ is nef over $S$.  
\end{cor}

\begin{proof}
If $K_X+\Delta$ is nef over $S$, then 
$(K_X+\Delta)|_{\Nklt(X, \Delta)}$ is obviously nef over 
$S$. 
Therefore, it is sufficient to construct 
a non-constant morphism 
$f:\mathbb A^1\to X\setminus \Nklt(X, \Delta)$ such that 
$\pi\circ f(\mathbb A^1)$ is a point under the 
assumption that 
$(K_X+\Delta)|_{\Nklt(X, \Delta)}$ is nef over $S$ and that 
$K_X+\Delta$ is not nef over $S$. 
By shrinking $S$ suitably, 
we may assume that 
$X$ and $S$ are both quasi-projective. 
By Lemma \ref{v-lem7.2}, we can construct a 
projective birational morphism 
$g:Y\to X$ from a normal $\mathbb Q$-factorial 
variety $Y$ satisfying 
(i), (ii), and (iv) in Lemma \ref{v-lem7.2}. 
Let us consider $\pi\circ g: Y\to S$. 
Note that $K_Y+\Delta_Y$ is not nef over $S$ 
since $K_Y+\Delta_Y=g^*(K_X+\Delta)$. 
It is obvious that $(K_Y+\Delta_Y)|_{\Nklt(Y, \Delta_Y)}$ is nef 
over $S$ by (iv) because so is $(K_X+\Delta)|_{\Nklt(X, \Delta)}$. 
Therefore, by Proposition \ref{v-prop7.1}, 
we have a non-constant morphism 
$h:\mathbb A^1\to Y\setminus \Nklt (Y, \Delta_Y)$ such that 
$(\pi\circ g)\circ h(\mathbb A^1)$ is a point. 
By Proposition \ref{v-prop7.1}, 
we have $0<-(K_Y+\Delta_Y)\cdot C\leq 2\dim Y=2\dim X$, 
where $C$ is the closure of $h(\mathbb A^1)$ in $Y$. 
Since $K_Y+\Delta_Y=h^*(K_X+\Delta)$ holds, $g$ does not 
contract $C$ to a point. 
This implies that 
$f:=g\circ h: \mathbb A^1\to X\setminus \Nklt(X, \Delta)$ 
is a desired non-constant morphism 
such that $\pi\circ f(\mathbb A^1)$ is a point by 
(iv). 
\end{proof}

We introduce the notion of {\em{open qlc strata}} 
in order to state the main result of this section 
(see Theorem \ref{v-thm7.5} below). 

\begin{defn}[Open qlc strata]\label{v-def7.4}
Let $W$ be a qlc stratum of a quasi-log canonical 
pair $[X, \omega]$. 
We put 
$$
U:=W\setminus \bigcup _{W'} W', 
$$ 
where $W'$ runs over qlc strata of $[X, \omega]$ strictly 
contained in $W$, and 
call it the {\em{open qlc stratum}} 
of $[X, \omega]$ associated to $W$. 
\end{defn}

The following theorem is the main result of this section, which 
is a generalization of \cite[Theorem 1.1]{svaldi} (see 
also \cite{lz}). 

\begin{thm}[{cf.~\cite[Theorem 1.1]{svaldi}}]\label{v-thm7.5}
Let $[X, \omega]$ be a quasi-log canonical pair 
and let $\pi:X\to S$ be a projective morphism 
onto a scheme $S$. 
Assume that for all open qlc strata $U$ of 
$[X, \omega]$ there is no non-constant morphism 
$f:\mathbb A^1\to U$ such that 
$\pi\circ f(\mathbb A^1)$ is a point. 
Then $\omega$ is nef over $S$. 
\end{thm}

\begin{proof}We divide the proof into five small steps. 

\begin{step}\label{v-7.5step1} 
We use induction on $\dim X$. 
If $\dim X=0$, then the statement is obvious. 
\end{step}

\begin{step}\label{v-7.5step2}
Let $X=\bigcup_{i\in I} X_i$ be the irreducible decomposition. 
Then $X_i$ is a qlc stratum of $[X, \omega]$ for 
every $i\in I$. 
By adjunction (see \cite[Theorem 6.3.5 (i)]{fujino-foundations}), 
$[X_i, \omega|_{X_i}]$ is a quasi-log canonical pair 
for every $i\in I$. 
We note that 
the qlc strata of $[X_i, \omega|_{X_i}]$ are 
exactly the qlc strata of $[X, \omega]$ 
contained in $X_i$ (see \cite[Theorem 6.3.5 (i)]{fujino-foundations}). 
Therefore, by replacing $[X, \omega]$ with $[X_i, \omega|_{X_i}]$, 
we may assume that $X$ is irreducible. 
\end{step}

\begin{step}\label{v-7.5step3}
By adjunction (see \cite[Theorem 6.3.5 (i)]{fujino-foundations}), 
$[\Nqklt (X, \omega), \omega|_{\Nqklt(X, \omega)}]$ becomes a quasi-log 
canonical pair whose qlc strata are exactly 
the qlc strata of $[X, \omega]$ 
contained in $\Nqklt(X, \omega)$. 
Therefore, by induction, 
$\omega|_{\Nqklt(X, \omega)}$ is nef over $S$. 
Therefore, it is sufficient to prove that $\omega$ is nef 
over $S$ under the assumption that $\omega|_{\Nqklt(X, \omega)}$ 
is nef over $S$. 
\end{step}

\begin{step}\label{v-7.5step4}
We take the normalization $\nu:X^\nu\to X$. 
Then, by \cite[Theorem 1.1]{fujino-haidong}, $[X^\nu, \nu^*\omega]$ 
naturally becomes a quasi-log canonical 
pair such that $\nu^{-1}\Nqklt(X, \omega)=
\Nqklt(X^\nu, \nu^*\omega)$ holds. 
We note that $\omega$ is nef over $S$ if and only if so is $\nu^*\omega$. 
\end{step}

\begin{step}\label{v-7.5step5}
We assume that $\omega$ is not nef over $S$. 
Without loss of generality, 
we may assume that $S$ is quasi-projective by shrinking 
$S$ suitably. 
Therefore, 
$X$ and $X^\nu$ are both quasi-projective. 
We take an ample $\mathbb Q$-divisor 
$H$ on $X^\nu$ such that 
$\nu^*\omega+H$ is not nef over $S$. 
By Theorem \ref{a-thm1.1}, 
we can take a boundary $\mathbb R$-divisor 
$\Delta$ on $X^\nu$ such that 
$K_{X^\nu}+\Delta\sim _{\mathbb R} \nu^*\omega+H$ and 
that $\Nklt (X^\nu, \Delta)=\Nqklt(X^\nu, \nu^*\omega) 
=\nu^{-1}\Nqklt(X, \omega)$. Thus $(K_{X^\nu}+\Delta)|_{\Nklt(X^\nu, 
\Delta)}$ is ample over $S$. Hence it is obviously 
nef over $S$. 
Since $K_{X^\nu}+\Delta$ is not nef over $S$, 
there exists a non-constant 
morphism $f:\mathbb A^1\to X^\nu \setminus \Nklt(X^\nu, \Delta)$ 
such that 
$(\pi\circ \nu)\circ f(\mathbb A^1)$ is a point by 
Corollary \ref{v-cor7.3}. 
Thus $\nu\circ f: \mathbb A^1
\to X\setminus \Nqklt(X, \omega)$ is a non-constant 
morphism 
such that $\pi\circ (\nu\circ f)(\mathbb A^1)$ is a point. 
This is a contradiction because 
$X\setminus \Nqklt(X, \omega)$ is an open 
qlc stratum of $[X, \omega]$. 
Therefore, $\omega$ is nef over $S$. 
\end{step} 
This is what we wanted. 
\end{proof}

As an obvious corollary of Theorem \ref{v-thm7.5}, we have: 

\begin{cor}\label{v-cor7.6}Let $[X, \omega]$ be a projective 
quasi-log canonical pair. 
Assume that $[X, \omega]$ is Mori hyperbolic, that is, 
for any open qlc stratum $U$, there is no non-constant 
morphism $f:\mathbb A^1\to U$. 
Then $\omega$ is nef. 
\end{cor}

We give a slight generalization of the cone theorem 
for quasi-log canonical pairs. For log canonical 
pairs, it is nothing but \cite[Theorem 1.2]{svaldi}. 
Of course, Theorem \ref{v-thm7.5} can be seen as a special 
case of Theorem \ref{v-thm7.7}. 

\begin{thm}[Cone theorem for quasi-log canonical 
pairs]\label{v-thm7.7}
Let $[X, \omega]$ be a quasi-log canonical pair and let $\pi:X\to S$ 
be a projective morphism onto a scheme $S$. 
Then we have 
$$
\overline{NE}(X/S)=\overline{NE}(X/S)_{\omega\geq 0} +\sum _j R_j, 
$$ 
where 
\begin{itemize}
\item[(i)] $R_j$ is spanned by a rational curve $C_j$ such that 
$\pi(C_j)$ is a point with 
$$
0<-\omega\cdot C_j \leq 2 \dim X, 
$$ 
and 
\item[(ii)] there exists an open qlc stratum $U$ of $[X, \omega]$ such that  
$C_j\cap U$ contains the image of a non-constant 
morphism $f:\mathbb A^1\to U$. 
\end{itemize}
\end{thm}

\begin{proof}[Sketch of Proof of Theorem \ref{v-thm7.7}]
In this proof, we only explain how to 
modify the proof of Theorem \ref{a-thm1.8}. 
So we will use the same notation as in the proof of 
Theorem \ref{a-thm1.8}. 
By construction, $(K_X+\Delta_\varepsilon)|_{\Nklt(X, \Delta_\varepsilon)}$ 
is obviously nef over $S$ since 
$\pi: \Nqklt(X, \omega)\to \pi(\Nqklt(X, \omega))$ is 
finite and $\Nklt(X, \Delta_\varepsilon)=
\Nqklt(X, \omega)$. 
Therefore, by the proof of Corollary \ref{v-cor7.3} 
(see also Proposition \ref{v-prop7.1}), 
we can take a non-constant morphism 
$f_\varepsilon: \mathbb A^1\to X\setminus \Nqklt(X, \omega)$ such that 
$C_\varepsilon$, the closure of $f_\varepsilon(\mathbb A^1)$, 
is a rational curve on $X$ such that 
$\pi(C_\varepsilon)$ is a point and that 
$0<-(K_X+\Delta_\varepsilon)\cdot C_\varepsilon\leq 2\dim X$. 
As in the proof of Theorem \ref{a-thm1.8}, 
we finally get a rational curve $C_j$ spanning $R_j$ with 
the desired properties. 
\end{proof}

We close this section with a sketch of the proof of Theorem \ref{a-thm1.10}. 

\begin{proof}[Sketch of Proof of Theorem \ref{a-thm1.10}]
We note that the cone and contraction theorem holds 
for semi-log canonical pairs by 
\cite[Theorem 1.19]{fujino-slc}. 
Let $R$ be a $(K_X+\Delta)$-negative extremal ray. By replacing 
$\pi:X\to S$ with the extremal contraction $\varphi_R: X\to Y$ over $S$ 
associated to $R$ and shrinking $S$ suitably, we may assume that 
$-(K_X+\Delta)$ is $\pi$-ample and that $X$ and $S$ are quasi-projective. 
Then, by Theorem \ref{p-thm2.9}, $[X, K_X+\Delta]$ naturally 
becomes a quasi-log canonical pair such that 
$V$ is a qlc stratum of $[X, K_X+\Delta]$ if and only if it is 
an slc stratum of $(X, \Delta)$. 
By the above sketch of the 
proof of Theorem \ref{v-thm7.7}, we can check that 
there exists a possibly singular rational curve $C$ with the desired 
properties. 
\end{proof}

Alternatively, in Theorem \ref{a-thm1.10}, 
we can take the normalization of $X$ and 
reduce the problem to the case where $(X, \Delta)$ is log canonical. 
Then we can apply \cite[Theorems 1.2 and 6.5]{svaldi} to find a desired 
rational curve $C$. 

\end{document}